\documentclass[11pt]{article}
\usepackage{amsmath,amsthm,verbatim,amssymb,amsfonts,amscd, graphicx, fancyhdr, enumerate}
\usepackage{parskip}

\usepackage[utf8]{inputenc} 
\usepackage[T1]{fontenc} 

\makeatletter
\def\thm@space@setup{%
	\thm@preskip=\parskip \thm@postskip=0pt
}
\makeatother 

\theoremstyle{plain} \numberwithin{equation}{section}
\newtheorem{theorem}{Theorem}[section]
\newtheorem{corollary}[theorem]{Corollary}

\newtheorem*{conjecture*}{Conjecture}
\newtheorem{lemma}[theorem]{Lemma}

\newtheorem*{question*}{Question} \topmargin-2cm
\theoremstyle{definition}
\newtheorem{definition}[theorem]{Definition}

\newtheorem{remark}[theorem]{Remark}

\usepackage[textheight=23cm, textwidth=16.5cm, footskip=1.0cm]{geometry}

\DeclareMathOperator{\Hessian}{Hess}
\DeclareMathOperator{\distance}{dist}

\renewcommand{\Re}{\operatorname{Re}}

\usepackage[
backend=bibtex,
natbib=true,
url=false, 
]{biblatex}
\begin{document}
	
	\title{The intrinsic geometry on bounded pseudoconvex domains}

	\author{
		Bingyuan Liu\\ bingyuan@ucr.edu
	}

	\date{\today}
	
\makeatletter
\newcommand{\subjclass}[2][1991]{%
	\let\@oldtitle\@title%
	\gdef\@title{\@oldtitle\footnotetext{#1 \emph{Mathematics subject classification.} #2}}%
}
\newcommand{\keywords}[1]{%
	\let\@@oldtitle\@title%
	\gdef\@title{\@@oldtitle\footnotetext{\emph{Key words and phrases.} #1.}}%
}
\makeatother

	
		\maketitle
	
	\begin{abstract}
		The Diederich--Forn\ae ss index has been introduced since 1977 to classify bounded pseudoconvex domains. In this article, we derive several intrinsic, geometric conditions on boundary of domains for  arbitrary indexes. Many results, in the past, by various mathematicians estimated the index by assuming some properties of domains. Our motivation of this paper is, the other way around, to look for how the index effects properties and shapes of domains. Especially, we look for a necessary condition of all bounded pseudoconvex domains $\Omega\subset\mathbb{C}^2$ with the Diederich--Forn\ae ss index 1. We also show that, when the Levi-flat set of $\partial\Omega$ is a closed Riemann surface, then the necessary condition can be simplified. 
	\end{abstract}
	
	\section{Introduction}\label{sec0}
	
	Let $\Omega$ be a bounded domain in $\mathbb{C}^n$ with smooth boundary and \[\delta(z):=\begin{cases}
	-\distance(z, \partial\Omega) & z\in\Omega\\
	\distance(z, \partial\Omega) & \text{otherwise}.
	\end{cases}\] The $\Omega$ is said to be pseudoconvex if $-\log(-\delta(z))$ is plurisubharmonic in $\Omega$. Note that $-\log(-\delta(z))$ blows up whenever $z$ approaches the boundary $\partial\Omega$. Indeed, all bounded pseudoconvex domains with $C^2$ boundary admit bounded (strictly) plurisubharmonic functions which vanishes on $\partial\Omega$ was shown by Diederich--Forn\ae ss \citep{DF77b}. They proved that any relatively compact pseudoconvex domain in Stein manifolds admits a defining function $\rho$ such that $-(-\rho)^\eta$ is (strictly) plurisubharmonic in $\Omega$ for some $\eta\in(0, 1]$. The author also remark that the $\rho$ may not be $\delta$ in general. For the pseudoconvex domain in complex manifolds, see Range \citep{Ra81}.
	
	For the bounded pseudoconvex domain $\Omega$ with $C^1$ boundary, Kerzman--Rosay constructed a smooth (strictly) plurisubharmonic function in \citep{KR81}. This function also vanishes on the boundary $\partial\Omega$. Later, Demailly in \citep{De87} improved the result of Kerzman--Rosay to arbitrary bounded pseudoconvex domains in $\mathbb{C}^n$ with Lipschitz boundary. He also showed his plurisubharmonic smooth function $\phi$ is bounded above and below by a multiple of $-\frac{1}{\log(-\delta)}$ near the boundary. In \citep{Ha08}, Harrington found a new smooth (strictly) plurisubharmonic $\phi$ in bounded pseudoconvex domains in $\mathbb{C}^n$ with Lipschitz boundary. In this paper, the $\phi$ is H\"{o}lder continuous on the boundary $\partial\Omega$. In \citep{Ha15}, he also obtained some results about pseudoconvex domain with Lipschitz boundary in $\mathbb{CP}^n$. 
	
	If the boundary is smooth or at least $C^2$, then the above result of Diederich--Forn\ae ss is applicable. In particular, the fact that $-(-\rho)^{\eta_0}$ is (strictly) plurisubharmonic will implies that $-(-\rho)^{\eta}$ is (strictly) plurisubharmonic for all $0<\eta<\eta_0$. It is of great interest, on a given domain to optimize the exponent in $-(-\rho)^\eta$ of plurisubharmonicity. 
	
	We now introduce the Diederich--Forn\ae ss index.
	
	\begin{definition}\label{df}
		Let $\Omega$ be a bounded, pseudoconvex domain in $\mathbb{C}^n$. The number 
		$0 < \tau_\rho < 1$ is called a {\it Diederich--Forn\ae ss exponent} if there exists a smooth defining 
		function $\rho$ of $\Omega$ so that $-(-\rho)^{\tau_\rho}$ is plurisubharmonic. 
		The index 
		$$
		\eta:=\sup \tau_\rho \, ,
		$$
		is 
		called the {\it Diederich--Forn\ae ss index} of the domain $\Omega$, where the supremum is taken over all smooth defining functions of $\Omega$.
	\end{definition}
	
	The Diederich--Forn\ae ss index is independent of the defining functions and only depends on the domain $\Omega$ itself. It is also known to be closely related to the regularity properties of $\overline{\partial}$-Neumann problem and the existence of Stein neighborhood basis. For these topics, see the book written by Straube \citep{St10} and the survey of Boas-Straube \citep{BS99}.
	
	The papers of Gallagher--McNeal in \citep{HM12a} and \citep{HM12b} include some interesting results in the fashion of Diederich--Forn\ae ss index for real variables. (Gallagher was previously known as Herbig.)
	
	The Diederich--Forn\ae ss index can be a number on $(0, 1]$. Whether the index reflects geometric boundary properties is not understood and in fact results of this nature seem to be unknown.  \c{S}ahuto\u{g}lu--Straube \citep{SS06} considered the compactness of the $\overline{\partial}$-Neumann operator in the similar fashion. 
	
	We briefly mention a few motivations about our study of intrinsic geometry related to the Diederich--Forn\ae ss index. In \citep{FH07}, Forn\ae ss--Gallagher proved a sufficient condition of the Diederich--Forn\ae ss index being 1. Their theorem asserts that if a defining function of a given domain is plurisubharmonic on the boundary, then the Diederich--Forn\ae ss index is 1. However, the converse is known not to be true. That is, if a domain has Diederich--Forn\ae ss index 1, it does not necessarily admit a defining function which is plurisubharmonic on the boundary. The counterexample was given in Behrens \citep{Be85}. 
	The theorem of Forn\ae ss--Gallagher in \citep{FH07} has been recently extended by Krantz--Liu--Peloso in \citep{KLP16}, where they showed that the sufficient condition of the Diederich--Forn\ae ss index being 1 is that, $\Hessian_\rho(L, N)=0$  on Levi-flat sets for some defining function $\rho$. Their theorem covers the example of Behrens mentioned above. However, it is still not known if $\Hessian_\rho(L, N)=0$ on Levi-flat sets is a equivalent condition for the domain admitting the Diederich--Forn\ae ss index 1. Inspired by these, we study the necessary condition of a specific Diederich-Forn\ae ss index.
	
	Mathematicians also obtained necessary conditions of the Diederich--Forn\ae ss exponents. For example, Kohn in \citep{Ko99} proved a result about the implication of the Diederich--Forn\ae ss exponent related to the boundness of the Bergman projection. Moreover, in 2000, Berndtsson--Charpentier showed in \citep{BC00} that the Bergman projection $P$ on $\Omega_\beta$ does map Sobolev space $W^k(\Omega_\beta)$ into $W^\kappa(\Omega_\beta)$ when $\kappa< \tau/2$ where $\tau$ is a Diederich--Forn\ae ss exponent. With a different point of view, we hope to discover a necessary condition on the geometry of domains. This geometry should be understood as a intrinsic property attached to boundary of domains.
	
	Motivated by the result of Krantz--Liu--Peloso, where one can see that $\frac{\|\nabla\rho\|}{\Hessian_\rho(L, N)}=\infty$ implies the index is 1, we raise the following natural questions:
	
	\begin{question*}
		\begin{enumerate}
			\item Let $\eta_0\in(0, 1]$. Can one find a necessary condition for the Diederich--Forn\ae ss index being $\eta_0$. The conditions should be intrinsic to the Levi-flat sets.
			\item What is the necessary condition for the case $\eta_0=1$? It is also interesting to compare it with the sufficient condition that $\frac{\|\nabla\rho\|}{\Hessian_\rho(L, N)}=\infty$ on Levi-flat sets. If they are different, which additional condition can be added to make they look similar.
		\end{enumerate}
	\end{question*}

	In this article, we consider these questions from the viewpoint of geometric analysis. To answer Question 1, we prove two main theorems as what follows. They are particularly useful for the case that Diederich--Forn\ae ss index is 1. The first necessary condition reveals how the term $\frac{\|\nabla\rho\|}{\Hessian_\rho(L, N)}$ behaves on Levi-flat sets for defining functions $\rho$ of $\Omega$. We sometimes call this term \textit{torsion} throughout the article. Here the notation of $\Hessian_\rho(L, N)$ denotes the complex Hessian on the direction of $L$ and $N$ (see Section \ref{sec1} for the definition of $L$ and $N$). The proof can be found in Section \ref{1st} and the definitions of notations can be found in Section \ref{sec1}.
	
	\begin{theorem}[The first necessary condition]\label{first condition}
		Let $\Omega$ be a bounded pseudoconvex domain with smooth boundary in $\mathbb{C}^2$ and $\Sigma$ be the Levi-flat set of $\partial\Omega$. If the Diederich--Forn\ae ss index is $\eta_0$, then for any $\eta<\eta_0$, there exists a smooth defining functions $\rho$ such that \[ 2.5+3.75C\frac{\|\nabla\rho\|}{\left|\Hessian_{\rho}(N, L)\right|}+0.5\left|L\left(\frac{\|\nabla\rho\|}{\Hessian_{\rho}(N, L)}\right)\right|\geq \frac{1}{1-\eta}\]on $\Sigma$, where 
		\[		C:=\max\lbrace
		\max_{z\in\Sigma}|g(\nabla_N\overline{L}, \overline{N})|,
		\max_{z\in\Sigma}|g(\nabla_{\overline{L}} L, L)|,
		\max_{z\in\Sigma}|g([N, L], N)|,
		\max_{z\in\Sigma}|g(\nabla_{\overline{L}}N, N)|\rbrace.\]
		
		Particularly, if the Diederich--Forn\ae ss index is 1, then there exists a sequence of defining functions $\rho_j$ so that $\frac{\|\nabla\rho_j\|}{\Hessian_{\rho_j}(N, L)}$ diverges to $\infty$ on $\Sigma$ in the holomorphic $C^1$ norm.
	\end{theorem}
	Here, $\|\nabla\rho\|^2$ denotes $g(\nabla\rho,\nabla\rho)$ where $g$ is the Euclidean metric.
	
	\begin{remark}
		\begin{enumerate}
			\item $C=C(\Sigma)$ is a constant depending on the Levi-flat set $\Sigma$ of $\partial\Omega$. 
			\item It is not difficult to see that $\frac{\|\nabla\rho\|}{\Hessian_{\rho}(N, L)}$ is of scaling invariance. That is, \[\frac{\|\nabla(K\cdot\rho)\|}{\Hessian_{K\cdot\rho}(N, L)}=\frac{\|\nabla\rho\|}{\Hessian_{\rho}(N, L)}\] for $K>0$. Indeed, $\frac{\|\nabla\rho\|}{\Hessian_{\rho}(N, L)}$ plays a key role in Theorem \ref{1st}. It is also critical in \citep{KLP16}.
			\item Combining Theorem \ref{first condition} with the results in \citep{KLP16}, we see that the Diederich--Forn\ae ss index being 1 essentially means how well the torsion is.
		\end{enumerate}
	\end{remark}
	
	The second necessary condition looks more deeply into the defining functions $\rho$ of $\Omega$. Indeed, each of defining functions $\rho$ can be written as $\rho=\delta e^\psi$, where $\psi$ is an arbitrary smooth function. Different $\psi$ gives a different defining function $\rho$ and hence a different Diederich--Forn\ae ss exponent. The second necessary condition aims at revealing the way that how $\psi$ effects its Diederich--Forn\ae ss exponents. The proof can be found in Section \ref{2nd}.
	
	\begin{theorem}[The second necessary condition]\label{second condition}
		Let $\Omega$ be a bounded pseudoconvex domain with smooth boundary in $\mathbb{C}^2$ and $\Sigma$ be the Levi-flat set of $\partial\Omega$. If the Diederich--Forn\ae ss index is $\eta_0$, then for any $\eta<\eta_0$, there is a real smooth function $\psi$ defined on a neighborhood of $\Sigma$ such that on all points of $\Sigma$ either
		\[\frac{1}{1-\eta}-1\leq\frac{C_2-\frac{1}{4}\Hessian_\psi(L, L)}{|\frac{1}{2}\overline{L}\psi+\Hessian_\delta(N, L)|^2}+\frac{C_1}{|\frac{1}{2}\overline{L}\psi+\Hessian_\delta(N, L)|}\] or\[|\frac{1}{2}\overline{L}\psi+\Hessian_\delta(N, L)|=0,\]
		where 
		\[C_1:=2C+2\max_{\Sigma}|\Hessian_\delta(N, L)|
		\quad \text{	and }\quad	C_2:=\frac{1}{2}\max_{\Sigma}\left(\Re(-L\Hessian_\delta(N,L)+\Hessian_\delta(N, \nabla_{\overline{L}}L))\right).\]
	\end{theorem}
	
	\begin{remark}
		$C_1, C_2$ are two constants depending on the Levi-flat set $\Sigma$ of the domain $\Omega$. 
	\end{remark}
	
	To Question 2 for an application of our necessary conditions, we first enhance the second necessary condition as what follows.

	\begin{theorem}\label{improv2nd}
		Let $\Omega$ be a bounded pseudoconvex domain with smooth boundary in $\mathbb{C}^2$ and $\Sigma$ be the Levi-flat set of $\partial\Omega$. If the Diederich--Forn\ae ss index is 1, then there is a family of real smooth functions $\psi_n$ defined on a neighborhood of $\Sigma$ such that on all points of $\Sigma$, 
		\[0\leq	\left|\frac{1}{2}\overline{L}\psi_n+g(\nabla_N\nabla\delta, L)\right|\leq\frac{C_1+\sqrt{n}\left(1+\frac{C_1^2}{n}+\left(4C_2-\Hessian_{\psi_n}(L, L)\right)\right)}{2n}\] for all $n\in\mathbb{N}$.
	\end{theorem}

	The theorem above is of independent interest and is a practical result. As an application of it, one obtains the following  theorem which gives a necessary condition for the case that $\Sigma$ is closed Riemann surface.

	\begin{theorem}\label{app}
		Let $\Omega$ be a bounded pseudoconvex domain with smooth boundary in $\mathbb{C}^2$ and $\Sigma$ be the Levi-flat set of $\partial\Omega$. Assume that $\Sigma$ is a closed Riemann surface. If the Diederich--Forn\ae ss index is 1, then there exists a sequence of smooth defining functions $\rho_j$ so that $\frac{\Hessian_{\rho_j}(N, L)}{\|\nabla\rho_j\|}\rightarrow 0$ in $L^1$-norm on $\Sigma$ as $j\rightarrow\infty$.
	\end{theorem}
	
	\begin{remark}
		For the boundness of Bergman projections and $\bar{\partial}$-Neumann projections, $\Hessian_\rho(N, L)$ has been intensely studied by Boas--Straube in \citep{BS93} and \citep{BS91} and Straube--Sucheston in \citep{SS02} and \citep{SS03} for the regularity results. Theorem \ref{app} is a theorem of the Diederich--Forn\ae ss index. It should lead to some interesting connections to the Bergman projections and $\bar{\partial}$-Neumann projections as well. More insights will be clarified in the forthcoming articles.
	\end{remark}
	
	To conclude the introduction, we remind readers with the works on the Diederich--Forn\ae ss index away from $1$. In 1977, Diederich--Forn\ae ss found a class of domains called the $\beta$-worm domain in \citep{DF77a} which gives a non-trivial Diederich--Forn\ae ss index (i.e., an index strictly between
	0 and 1). 
	
	\begin{definition}[$\beta$-worm domain]
		Let $\beta>\pi/2$ and $\eta: \mathbb{R}\mapsto\mathbb{R}$ be a fixed smooth function with the following properties:
		\begin{enumerate}
			\item $\eta(x)\geq 0$, $\eta$ is even and convex.
			\item $\eta^{-1}(0)=I_{\beta-\pi/2}=[-\beta+\pi/2, \beta-\pi/2]$.
			\item there exists and $a>0$ such that $\eta(x)>1$ if $x<-a$ or $x>a$.
			\item $\eta'(x)\neq 0$ if $\eta(x)=1$.
		\end{enumerate}
		Then \[\Omega_\beta=\lbrace(z,w)\in\mathbb{C}^2: |z+e^{i\log|w|^2}|^2<1-\eta^2(\log|w|)\rbrace.\] is called a $\beta$-worm domain.
	\end{definition}
	
	In fact, the Diederich--Forn\ae ss index can be arbitrarily close to $0$ by increasing $\beta$ due to the following works. In 1992, Barrett showed in \citep{Ba92}, that the Bergman projection $P$ on $\Omega_\beta$ does not map the Sobolev space $W^\kappa(\Omega_\beta)$ into $W^\kappa(\Omega_\beta)$ when $\kappa\geq\pi/(2\beta-\pi)$. By \citep{BC00} mentioned above, the Diederich--Forn\ae ss index of $\Omega_\beta$ is less or equal to
	$2\pi/(2\beta-\pi)$. The reader can also verify this result from Krantz--Peloso \citep{KP08}. Indeed, Theorem 6 in \citep{DF77a} says that if the standard defining function of $\Omega_\beta$ has exponent $\leq\eta$, then all other defining functions have exponent $\leq\eta$, that is, the Diederich--Forn\ae ss index of $\Omega_\beta\leq\eta$. Thus, the calculation in \citep{KP08} shows that the index of  $\Omega_\beta\leq 2\pi/(2\beta-\pi)$. 
	
	Recently Fu--Shaw and Adachi--Brinkschulte proved independently in \citep{FS14} and \citep{AB14} respectively that, roughly speaking, if a relatively compact domain in a $n$-dimensional complex manifold has all boundary points Levi-flat, then the Diederich--Forn\ae ss index cannot be greater than $1/n$.
	
	\section{Preliminaries}\label{sec1}
	We remind the readers with some basic notations in several complex variables and complex geometry. Here $z$ should be read as $z_j$ and $x, y$ should be read as $x_j, y_j$ in case it involves several variables.
	\[\frac{\partial}{\partial x}=\frac{\partial}{\partial z}+\frac{\partial}{\partial \bar{z}}\]
	\[\frac{\partial}{\partial y}=i\frac{\partial}{\partial z}-i\frac{\partial}{\partial \bar{z}}\]
	\[\frac{\partial}{\partial z}=\frac{1}{2}(\frac{\partial}{\partial x}-i\frac{\partial}{\partial y})\]
	\[\frac{\partial}{\partial \bar{z}}=\frac{1}{2}(\frac{\partial}{\partial x}+i\frac{\partial}{\partial y})\]
	
	Let $f$ be a smooth function. We have that,
	\[\left|\frac{\partial f}{\partial z}\right|^2+\left|\frac{\partial f}{\partial \bar{z}}\right|^2=\frac{1}{4}\left(\left|\frac{\partial f}{\partial x}\right|^2+\left|\frac{\partial f}{\partial y}\right|^2\right)\]
	
	Throughout the article, we will use the following notations.
	
	\begin{definition}
		Let $\Omega$ be a bounded domain with smooth boundary in $\mathbb{C}^2$ defined by a smooth defining function $\rho$. The vector field \[L_\rho=\frac{1}{\sqrt{|\frac{\partial \rho}{\partial z}|^2+|\frac{\partial \rho}{\partial w}|^2}}(\frac{\partial \rho}{\partial w}\frac{\partial}{\partial z}-\frac{\partial \rho}{\partial z}\frac{\partial}{\partial w})\] on $\partial\Omega$ is called the \textit{normalized holomorphic tangential vector field }, and \[N_\rho=\frac{1}{\sqrt{|\frac{\partial \rho}{\partial z}|^2+|\frac{\partial \rho}{\partial w}|^2}}(\frac{\partial \rho}{\partial\bar{z}}\frac{\partial}{\partial z}+\frac{\partial \rho}{\partial\bar{w}}\frac{\partial}{\partial w})\] on $\partial\Omega$ is called the \textit{normalized complex normal vector field}.
	\end{definition}
	
	Let us remind the reader that \[N_\rho+\overline{N}_\rho=2\Re N_\rho=\frac{\nabla\rho}{\|\nabla\rho\|},\] which is, at $\partial\Omega$, the normal vector of $\partial\Omega$. In case that the defining function is the special $\delta$ that I mentioned at the beginning, instead of writing $L_\delta$ and $N_\delta$, we use the notations $L$ and $N$ respectively. The reader can also check that $\sqrt{2}L$ and $\sqrt{2}N$ form an  orthonormal basis on holomorphic tangent space. Moreover, \[\begin{split}
	N\delta=&\frac{1}{\sqrt{|\frac{\partial \delta}{\partial z}|^2+|\frac{\partial \delta}{\partial w}|^2}}(\frac{\partial \delta}{\partial\bar{z}}\frac{\partial \delta}{\partial z}+\frac{\partial \delta}{\partial\bar{w}}\frac{\partial \delta}{\partial w})\\=&\sqrt{|\frac{\partial \delta}{\partial z}|^2+|\frac{\partial \delta}{\partial w}|^2}\\=&\frac{1}{2}\|\nabla\delta\|\\=&\frac{1}{2}.
	\end{split}\]
	Here, the last line equality is because of $\|\nabla\delta\|=1$, which is a consequence of the fact that $\delta$ is a distance function.  \[N\rho=N(\delta e^\psi)=\delta e^\psi N\psi+e^\psi N\delta=\delta e^\psi N\psi+\frac{e^\psi }{2}.\]
	On $\partial\Omega$, since $\delta=0$, we have
	\[N\rho=\frac{e^\psi}{2}.\]
	
	We also use the standard definition for the Hessian of a function $f$ on real tangent vector fields:
	\[\Hessian_f(X, Y)=g(\nabla_X\nabla f, Y)=Y(Xf)-(\nabla_YX)f,\] and for holomorphic tangent vectors we calculate as follows:
	\[\Hessian_f(Z, W)=g(\nabla_Z\nabla f, W)=Z(\overline{W}f)-\nabla_Z\overline{W}f=\overline{\Hessian_f(W,Z)}.\]

	For the following paragraph, we define a norm in space of functions called holomorphic $C^1$ norm.
	
	Recall that if a function on $[a, b]\subset\mathbb{R}$ is $C^1$, we can defined the $C^1$ norm as follows.
	\[\|f\|_{C^1}=\sup_{a\leq x\leq b}|f(x)|+\sup_{a\leq x\leq b}|f'(x)|.\]
	We will imitate it to define the CR version of $C^1$ norm, i.e., holomorphic $C^1$ norm.
	
	\begin{definition}
		Let $M$ be a pseudoconvex hypersurface in $\mathbb{C}^2$ and $L$ be the unit holomorphic vector field on $M$. Assume that $f$ is a complex valued function defined on $M$. We denote the holomorphic $C^1$ norm by \[\|f\|_{C^1, L}:=\|f\|_\infty+\|Lf\|_\infty,\] where \[\|\cdot\|_\infty=\sup_{M}|\cdot|.\]
	\end{definition}
	
	\begin{remark}
		The holomorphic $C^1$ norm can be defined on a subset $\tilde{M}$ of $M$. In this case $\|\cdot\|_\infty$ denotes $\sup_{\tilde{M}}|\cdot|$. And we call \[\|f\|_{C^1, L}:=\|f\|_\infty+\|Lf\|_\infty\] the holomorphic $C^1$ norm of $f$ on $\tilde{M}$.
	\end{remark}
	
	\section{The first necessary condition of Diederich--Forn\ae ss index}\label{1st}
	
	Let $\delta$ be defined in Section \ref{sec0}. We want to modify the defining function in order to seek the best for optimizing the Diederich--Forn\ae ss exponent. Put $\rho=\delta e^{\psi}$, where $\psi$ will be determined later. One can see that $\Hessian_{-(-\rho)^\eta }$ is positive definite at all points in $\Omega$ if and only if \[\Hessian_{-(-\rho)^\eta }(aL+bN,aL+bN)\] is positive at all $z\in\Omega$ for all complex numbers $a,b\in\mathbb{C}$. This implies that 
	\[\begin{split}
	&\Hessian_{-(-\rho)^\eta }(aL+bN,aL+bN)\\=&|a|^2\Hessian_{-(-\rho)^\eta}(L, L)+|b|^2\Hessian_{-(-\rho)^\eta }(N, N)+2\Re (a\bar{b}\Hessian_{-(-\rho)^\eta }(L, N))\\=&\eta (-\rho)^{\eta-1}\Bigg(|a|^2\Big(\Hessian_\rho (L, L)+\frac{1-\eta}{-\rho}|L\rho|^2\Big)+2\Re \Big(a\bar{b}\Big(\Hessian_\rho (L, N)+\frac{1-\eta}{-\rho}L\rho\cdot\overline{N}\rho\Big)\Big)\\&+|b|^2 \Big(\Hessian_\rho (N, N)+\frac{1-\eta}{-\rho}N(\rho)\overline{N} (\rho)\Big)\Bigg)
	\end{split}\]
	is positive for all $a,b$ and all $z\in\Omega$ if and only if 
	\[\begin{split}
	|a|^2\Big(\Hessian_\rho (L, L)+\frac{1-\eta}{-\rho}|L\rho|^2\Big)+&2\Re \Big(a\bar{b}\Big(\Hessian_\rho (L, N)+\frac{1-\eta}{-\rho}L\rho\cdot\overline{N}\rho\Big)\Big)\\+&|b|^2 \Big(\Hessian_\rho (N, N)+\frac{1-\eta}{-\rho}N(\rho)\overline{N} (\rho)\Big)
	\end{split}
	\]
	is positive for all $a,b\in\mathbb{C}$ and all $z\in\Omega$. 
	
	We are going to show the following lemma which states that above inequality is equivalent to the following inequality:
	\[\begin{split}
	|a|^2\Big(\Hessian_\rho (L, L)+\frac{1-\eta}{-\rho}|L\rho|^2\Big)&-2\left|a\bar{b}\Big(\Hessian_\rho (L, N)+\frac{1-\eta}{-\rho}L\rho\cdot\overline{N}\rho\Big)\right|\\&+|b|^2 \Big(\Hessian_\rho (N, N)+\frac{1-\eta}{-\rho}N(\rho)\overline{N} (\rho)\Big)> 0
	\end{split}\]
	for all $a,b$ and $z\in\Omega$.
	
	\begin{lemma}
		\[\begin{split}
		|a|^2\Big(\Hessian_\rho (L, L)+\frac{1-\eta}{-\rho}|L\rho|^2\Big)&+2\Re\left(a\bar{b}\Big(\Hessian_\rho (L, N)+\frac{1-\eta}{-\rho}L\rho\cdot\overline{N}\rho\Big)\right)\\&+|b|^2 \Big(\Hessian_\rho (N, N)+\frac{1-\eta}{-\rho}N(\rho)\overline{N} (\rho)\Big)> 0
		\end{split}\]
		for all $a,b$ and $z\in\Omega$ is equivalent to say \[\begin{split}
		|a|^2\Big(\Hessian_\rho (L, L)+\frac{1-\eta}{-\rho}|L\rho|^2\Big)&-2\left|a\bar{b}\Big(\Hessian_\rho (L, N)+\frac{1-\eta}{-\rho}L\rho\cdot\overline{N}\rho\Big)\right|\\&+|b|^2 \Big(\Hessian_\rho (N, N)+\frac{1-\eta}{-\rho}N(\rho)\overline{N} (\rho)\Big)> 0
		\end{split}\] for all $a,b$ and $z\in\Omega$.
	\end{lemma}
	
	\begin{proof}
		The sufficiency is the easy direction. This is because \[2\Re\left(a\bar{b}\Big(\Hessian_\rho (L, N)+\frac{1-\eta}{-\rho}L\rho\cdot\overline{N}\rho\Big)\right)\geq-2\left|a\bar{b}\Big(\Hessian_\rho (L, N)+\frac{1-\eta}{-\rho}L\rho\cdot\overline{N}\rho\Big)\right|.\] To show the necessity, we notice that at each $z\in\Omega$, we can find $a\in\mathbb{C}$ and $b\in\mathbb{C}$ such that
		\[a\bar{b}\Hessian_{-(-\rho)^\eta }(L, N)=-|a\bar{b}\Hessian_{-(-\rho)^\eta }(L, N)|,\]
		which completes the proof of necessity.
	\end{proof}
	
	This implies that $\Hessian_{-(-\rho)^\eta }(aL+bN,aL+bN)$ is positive if and only if \[\begin{split}
	\left|\frac{a}{b}\right|^2\Big(\Hessian_\rho (L, L)+\frac{1-\eta}{-\rho}|L\rho|^2\Big)&-2\left|\frac{a}{b}\right|\left|\Big(\Hessian_\rho (L, N)+\frac{1-\eta}{-\rho}L\rho\cdot\overline{N}\rho\Big)\right|\\&+ \Big(\Hessian_\rho (N, N)+\frac{1-\eta}{-\rho}N(\rho)\overline{N} (\rho)\Big)> 0
	\end{split}\] for all $a,b$ and $z\in\Omega$.
	
	To study the preceding inequality, we need some elementary preparations.
	
	\begin{lemma}
		\begin{equation}
		\label{prequa}\begin{split}
		\left|\frac{a}{b}\right|^2\Big(\Hessian_\rho (L, L)+\frac{1-\eta}{-\rho}|L\rho|^2\Big)-2\left|\frac{a}{b}\right|\left|\Big(\Hessian_\rho (L, N)+\frac{1-\eta}{-\rho}L\rho\cdot\overline{N}\rho\Big)\right|&\\+ \Big(\Hessian_\rho (N, N)+\frac{1-\eta}{-\rho}N(\rho)\overline{N} (\rho)\Big)&> 0
		\end{split}
		\end{equation} implies \begin{equation}\label{delta}\begin{split}
		\bigg|\Hessian_\rho (L, N)+&\frac{1-\eta}{-\rho}L\rho\cdot\overline{N}\rho\bigg|^2\\-&\left(Hess_\rho (L, L)+\frac{1-\eta}{-\rho}|L\rho|^2\right)\left(\Hessian_\rho (N, N)+\frac{1-\eta}{-\rho}N(\rho)\overline{N} (\rho)\right)\leq 0.
		\end{split}
		\end{equation}
		
	\end{lemma}
	\begin{proof}
		We consider the quadratic inequality \begin{equation}\label{qua}
		\begin{split}
		\Big(\Hessian_\rho (L, L)+\frac{1-\eta}{-\rho}|L\rho|^2\Big)\xi^2-2\left|\Big(\Hessian_\rho (L, N)+\frac{1-\eta}{-\rho}L\rho\cdot\overline{N}\rho\Big)\right|\xi&\\+ \Big(\Hessian_\rho (N, N)+\frac{1-\eta}{-\rho}N(\rho)\overline{N} (\rho)\Big)&> 0,
		\end{split}
		\end{equation} for indefinite word $\xi$. It is clear that \[\Hessian_\rho (L, L)+\frac{1-\eta}{-\rho}|L\rho|^2\geq 0\] because the domain $\Omega$ is pseudoconvex. Suppose that at $z\in\Omega$, $\Hessian_\rho (L, L)+\frac{1-\eta}{-\rho}|L\rho|^2>0$, the left side of (\ref{qua}) is a polynomial in $\xi$ of degree 2. Consider its axis of symmetry, \[\frac{\left|\Big(\Hessian_\rho (L, N)+\frac{1-\eta}{-\rho}L\rho\cdot\overline{N}\rho\Big)\right|}{\Big(\Hessian_\rho (L, L)+\frac{1-\eta}{-\rho}|L\rho|^2\Big)}\geq 0.\] We see that if the left side of (\ref{qua}) has a solution, then it must admit a positive solution. Since $\left|\frac{a}{b}\right|$ can be any positive number, if (\ref{prequa}) holds, then one must have that \[\begin{split}
		\Delta=\bigg|\Hessian_\rho (L, N)&+\frac{1-\eta}{-\rho}L\rho\cdot\overline{N}\rho\bigg|^2\\&-\left(Hess_\rho (L, L)+\frac{1-\eta}{-\rho}|L\rho|^2\right)\left(\Hessian_\rho (N, N)+\frac{1-\eta}{-\rho}N(\rho)\overline{N} (\rho)\right)<0.
		\end{split}
		\]
		
		Suppose that at $z\in\Omega$, $\Hessian_\rho (L, L)+\frac{1-\eta}{-\rho}|L\rho|^2=0$, then the left side of (\ref{qua}) is linear. Because \[-2\left|\Big(\Hessian_\rho (L, N)+\frac{1-\eta}{-\rho}L\rho\cdot\overline{N}\rho\Big)\right|\leq 0,\] (\ref{qua}) cannot hold for all positive $\xi$ unless that \[\left|\Big(\Hessian_\rho (L, N)+\frac{1-\eta}{-\rho}L\rho\cdot\overline{N}\rho\Big)\right|= 0\] at $z\in\Omega$ too. Since the assumption $\Hessian_\rho (L, L)+\frac{1-\eta}{-\rho}|L\rho|^2=0$, (\ref{prequa}), in this case, implies at $z$ \[\begin{split}
		\Delta=\bigg|\Hessian_\rho (L, N)&+\frac{1-\eta}{-\rho}L\rho\cdot\overline{N}\rho\bigg|^2\\&-\left(Hess_\rho (L, L)+\frac{1-\eta}{-\rho}|L\rho|^2\right)\left(\Hessian_\rho (N, N)+\frac{1-\eta}{-\rho}N(\rho)\overline{N} (\rho)\right)=0.
		\end{split}
		\] One can now complete the proof by combining the two cases.
	\end{proof}
	
	By calculation, we see the inequality (\ref{delta}) is	
	\begin{equation}\label{basic}
	\begin{split}
	&\left|\Hessian_\rho (L, N)\right|^2+\left|\frac{1-\eta}{-\rho}L\rho\cdot\overline{N}\rho\right|^2+2\frac{1-\eta}{-\rho}\Re\left(\Hessian_\rho (L, N)\cdot \overline{L}\rho\cdot N\rho\right)\\&-\left(\Hessian_\rho (L, L)+\frac{1-\eta}{-\rho}|L\rho|^2\right)\left(\Hessian_\rho (N, N)+\frac{1-\eta}{-\rho}\left|N(\rho)\right|^2\right)\\=&\left|\Hessian_\rho (L, N)\right|^2+2\frac{1-\eta}{-\rho}\Re\left(\Hessian_\rho (L, N)\cdot \overline{L}\rho\cdot N\rho\right)-\frac{1-\eta}{-\rho}|L\rho|^2\Hessian_\rho (N, N)\\&-\Hessian_\rho (L, L)\Hessian_\rho (N, N)-\Hessian_\rho (L, L)\frac{1-\eta}{-\rho}\left|N(\rho)\right|^2\leq 0.
	\end{split}
	\end{equation}
	
	Let $\Sigma$ denotes the Levi-flat set of $\partial\Omega$. Note that $\Hessian_\rho(L, L)$ vanishes on $\Sigma$ and $\Hessian_\rho(N,N)$ is bounded around $\Omega$. Hence the term $\Hessian_\rho (L, L)\Hessian_\rho (N, N)$ in the inequality above vanishes, as we let $z$ approaches $\Sigma$ along the normal direction. We will claim that the term $\frac{|L\rho|^2}{-\rho}$ in (\ref{basic}) vanishes as $z$ approaches $\partial\Omega$ along the normal direction too. Indeed, let $z_0\in\partial\Omega$, we have $L\rho$ vanishes at $z_0$, so $|L\rho(z)|^2=o(|z-z_0|^2)$, where $z-z_0$ is in the normal direction of $\partial\Omega$. However $\rho(z)$ is only $o(|z-z_0|)$ because by the assumption of defining functions, $\nabla\rho(z_0)$ is nonzero. Hence $\frac{|L\rho|^2}{-\rho}$ in (\ref{basic}) vanishes as $z$ approaches $\partial\Omega$ along the normal direction.
	
	We are ready to simplify (\ref{basic}) a little. Let $z\rightarrow z_0\in\partial\Omega$ along the normal direction to the Levi-flat set $\Sigma$, (\ref{basic}) is followed by
	\[\left|\Hessian_\rho (L, N)\right|^2\Big|_{z_0}\leq -2\lim\limits_{z\to z_0}\frac{1-\eta}{-\rho}\Re\left(\Hessian_\rho (L, N)\cdot \overline{L}\rho\cdot N\rho\right)+\lim\limits_{z\to z_0}\Hessian_\rho (L, L)\frac{1-\eta}{-\rho}\left|N(\rho)\right|^2,\]
	on $\Sigma$. 
	
	In case that
	\[\left|\Hessian_\rho (L, N)\right|\Big|_{z_0}\neq 0,\]
	
	we have	\[\frac{1}{1-\eta}\leq\frac{-2\Re\left(\Hessian_\rho (L, N)\cdot\frac{\overline{L}\rho}{-\rho}\cdot N\rho\right)+\frac{\Hessian_\rho (L, L)}{-\rho}\left|N(\rho)\right|^2}{\left|\Hessian_\rho (L, N)\right|^2},\]
	on $\Sigma$. 
	
	Hence, we get a condition for the Diederich--Forn\ae ss index being $\eta$. From now on, unless we remind the reader, the notation $z\rightarrow z_0\in\partial\Omega$, $z\rightarrow\partial\Omega$ and $z\rightarrow\Sigma$ mean the limits taken as points approaching in normal directions. We also skip writing the limit sometimes for concision. For example, $\lim\limits_{z\to z_0}\frac{\overline{L}\rho}{\rho}$ was written as $\frac{\overline{L}\rho}{\rho}$ in the preceding inequality.
	
	From the discussion above, we obtain the following theorem.	
	\begin{theorem}\label{rough}
		Let $\Omega$ be a bounded pseudoconvex domain with smooth boundary in $\mathbb{C}^2$ with Diederich--Forn\ae ss index $\eta_0$. Let $\Sigma$ be the Levi-flat sets of $\partial\Omega$. Then for any number $0<\eta<\eta_0$, there exists a smooth defining function $\rho$ such that at each point $z_0\in\Sigma$, 	\[\left|\Hessian_\rho (L, N)\right|^2\leq -2\frac{1-\eta}{-\rho}\Re\left(\Hessian_\rho (L, N)\cdot \overline{L}\rho\cdot N\rho\right)+\Hessian_\rho (L, L)\frac{1-\eta}{-\rho}\left|N(\rho)\right|^2,\]
		In particular, if $\left|\Hessian_{\rho} (L, N)\right|\neq 0$, \[\frac{1}{1-\eta}\leq\frac{-2\Re\left(\Hessian_{\rho} (L, N)\cdot\frac{\overline{L}\rho}{-\rho}\cdot N(\rho)\right)+\frac{\Hessian_{\rho} (L, L)}{-\rho}\left|N(\rho)\right|^2}{\left|\Hessian_{\rho} (L, N)\right|^2}.\] Otherwise,   \[\frac{\Hessian_{\rho} (L, L)}{-{\rho}}\geq 0.\]
	\end{theorem}
	
	We are going to prove the first necessary condition for Diederich--Forn\ae ss index. But before that, we define a quantity $C>0$.
	
	We let
	\[C:=\max\lbrace
	\max_{z\in\Sigma}|g(\nabla_N\overline{L}, \overline{N})|,
	\max_{z\in\Sigma}|g(\nabla_{\overline{L}} L, L)|,
	\max_{z\in\Sigma}|g([N, L], N)|,
	\max_{z\in\Sigma}|g(\nabla_{\overline{L}}N, N)|\rbrace.\]
	
	Such a $C$ is a real positive number because $\Sigma$ is compact and $|g(\nabla_N\overline{L}, \overline{N})|$, $|g(\nabla_{\overline{L}} L, L)|$, $|g([N, L], N)|$ and $|g(\nabla_{\overline{L}}N, N)|$ are continuous.
	
	We are now ready to prove the first necessary condition of Diederich--Forn\ae ss index.
	
	\begin{proof}[Proof of Theorem \ref{first condition}]
		We will only show the case that $|\Hessian_\rho(L, N)|\neq 0$ because otherwise, 
		\[\frac{\|\nabla\rho\|}{\left|\Hessian_\rho(L, N)\right|}=\infty,\] which implies that \[ 2.5+3.75C\frac{\|\nabla\rho\|}{\left|\Hessian_{\rho}(N, L)\right|}+0.5\left|L\left(\frac{\|\nabla\rho\|}{\Hessian_{\rho}(N, L)}\right)\right|\geq 3.75C\frac{\|\nabla\rho\|}{\left|\Hessian_\rho(L, N)\right|} \geq\frac{1}{1-\eta},\] for any $\eta$.
		
		Assume that $|\Hessian_\rho(L, N)|\neq 0$ and we observe that firstly
		\[\overline{L}\rho=0 \] on $\partial\Omega$ and $N-\overline{N}$ is a (real) vector field on $\partial\Omega$. This is because  \[(N-\overline{N})\rho=\sqrt{\left|\frac{\partial\rho}{\partial z}\right|^2+\left|\frac{\partial\rho}{\partial w}\right|^2}-\sqrt{\left|\frac{\partial\rho}{\partial z}\right|^2+\left|\frac{\partial\rho}{\partial w}\right|^2}=0.\] Because $N-\overline{N}$ is a vector field on $\partial\Omega$ and $\overline{L}\rho=0$ on $\partial\Omega$,
		\[(N-\overline{N})\overline{L}\rho=0\] on $\partial\Omega$. This means
		\[N(\overline{L}\rho)=\overline{N}(\overline{L}\rho).\]
		
		We calculate that
		\[\lim\limits_{z\to\partial\Omega}\frac{\overline{L}\rho}{-\rho}=\lim\limits_{z\to\partial\Omega}\frac{\overline{L}\rho|_z-\overline{L}\rho|_p}{-\rho(z)+\rho(p)}=\frac{(N+\overline{N})\overline{L}\rho}{-g(\nabla\rho, N+\overline{N})}=\frac{2N(\overline{L}\rho)}{-g(\nabla\rho, N+\overline{N})}=\frac{2\overline{\Hessian_\rho(L, N)}+2(\nabla_N \overline{L})\rho}{-\|\nabla\rho\|},\] where $p\in\partial\Omega$ is the closest point to $z$.
		
		Moreover, for the same reason, on the Levi-flat set $\Sigma$ \[N\Hessian_\rho(L, L)=\overline{N}\Hessian_\rho(L, L)\] and thus
		\[\begin{split}
		&(N+\overline{N})\Hessian_\rho(L, L)=2N\Hessian_\rho(L, L)=2Ng(\nabla_L\nabla\rho, L)\\=&2g(\nabla_L\nabla_N\nabla\rho, L)+2g(\nabla_L\nabla\rho, \nabla_{\overline{N}}L)+2g(\nabla_{[N, L]}\nabla\rho, L)\\=&2Lg(\nabla_N\nabla\rho, L)-2g(\nabla_N\nabla\rho, \nabla_{\overline{L}} L)+2g(\nabla_L\nabla\rho, \nabla_{\overline{N}}L)+2g(\nabla_{[N, L]}\nabla\rho, L).
		\end{split}\]
		
		From the calculation above, we can further get that
		\[\begin{split}
		&\lim\limits_{z\to\partial\Omega}\frac{\Hessian_\rho(L, L)}{-\rho}=\frac{(N+\overline{N})\Hessian_\rho(L, L)}{-(N+\overline{N})\rho}=\frac{2N\Hessian_\rho(L, L)}{-\|\nabla\rho\|}\\=&\frac{Lg(\nabla_N\nabla\rho, L)-g(\nabla_N\nabla\rho, \nabla_{\overline{L}} L)+g(\nabla_L\nabla\rho, \nabla_{\overline{N}}L)+g(\nabla_{[N, L]}\nabla\rho, L)}{-\|\nabla\rho\|}\\&+\frac{\overline{L}g(\nabla_L\nabla\rho, N)-g(\nabla_{\overline{N}}\nabla\rho, \nabla_L \overline{L})+g(\nabla_{\overline{L}}\nabla\rho, \nabla_{N}\overline{L})+g(\nabla_{\overline{[N, L]}}\nabla\rho, \overline{L})}{-\|\nabla\rho\|}\\
		\leq&\frac{2\Re(Lg(\nabla_N\nabla\rho, L))}{-\|\nabla\rho\|}+\frac{2\left|g(\nabla_N\nabla\rho, \nabla_{\overline{L}} L)\right|+2\left|g(\nabla_L\nabla\rho, \nabla_{\overline{N}}L)\right|+2\left|g(\nabla_{[N, L]}\nabla\rho, L)\right|}{\|\nabla\rho\|}.
		\end{split}\]
		
		Moreover, since on $\partial\Omega$, \[N\rho=\frac{e^\psi}{2}=\frac{1}{2}\|\nabla\rho\|\]
		
		\[\begin{split}
		&|N(\rho)|^2\lim\limits_{z\to\partial\Omega}\frac{\Hessian_\rho(L, L)}{-\rho}\\=&\left(-\frac{\|\nabla\rho\|}{2}\right)\left(\Re(Lg(\nabla_N\nabla\rho, L))-\Re(g(\nabla_N\nabla\rho, \nabla_{\overline{L}} L))+\Re(g(\nabla_L\nabla\rho, \nabla_{\overline{N}}L))+\Re(g(\nabla_{[N, L]}\nabla\rho, L))\right)\\
		\leq&-\frac{1}{2}\|\nabla\rho\|\Re\left(Lg(\nabla_N\nabla\rho, L)\right)+\frac{1}{2}\|\nabla\rho\|\left(\left|g(\nabla_N\nabla\rho, \nabla_{\overline{L}} L)\right|+\left|g(\nabla_L\nabla\rho, \nabla_{\overline{N}}L)\right|+\left|g(\nabla_{[N, L]}\nabla\rho, L)\right|\right),
		\end{split}\]
		and			
		\[\begin{split}
		&-2\Re\left(\Hessian_\rho(L, N)\cdot \lim\limits_{z\to\partial\Omega}\frac{\overline{L}\rho}{-\rho}\cdot N(\rho)\right)\\=&2\Re\left(\Hessian_\rho(L, N)\cdot\frac{\overline{\Hessian_\rho(L, N)}+(\nabla_N \overline{L})\rho}{N\rho}\cdot N\rho\right)\\=&2\Re\left(\Hessian_\rho(L, N)\cdot\left(\overline{\Hessian_\rho(L, N)}+(\nabla_N \overline{L})\rho\right)\right)\\=&2|\Hessian_\rho(L, N)|^2+2\Re\left(\Hessian_\rho(L, N)\cdot(\nabla_N \overline{L})\rho\right)\\\leq&2|\Hessian_\rho(L, N)|^2+2\left|\Hessian_\rho(L, N)\cdot(\nabla_N \overline{L})\rho\right|.
		\end{split}\]
		
		It implies that		
		\[\frac{1}{1-\eta}\leq2+\frac{2\frac{\left|\Hessian_\rho(L, N)\right|}{\|\nabla\rho\|}\frac{\left|(\nabla_N \overline{L})\rho\right|}{\|\nabla\rho\|}-\frac{\Re Lg(\nabla_N\nabla\rho, L)}{2\|\nabla\rho\|}+\frac{\left|g(\nabla_N\nabla\rho, \nabla_{\overline{L}} L)\right|}{2\|\nabla\rho\|}+\frac{\left|g(\nabla_L\nabla\rho, \nabla_{\overline{N}}L)\right|}{2\|\nabla\rho\|}+\frac{\left|g(\nabla_{[N, L]}\nabla\rho, L)\right|}{2\|\nabla\rho\|}}{\left(\frac{\left|\Hessian_\rho (L, N)\right|}{\|\nabla\rho\|}\right)^2}\]
		
		Combining the fact that \[0=\Hessian_\rho(L, L)=\overline{L}(L\rho)-(\nabla_{\overline{L}}L)\rho=(\nabla_{\overline{L}}L)\rho\] and the definition of $C$, we obtain that on $\Sigma$, 
		\[\begin{split}
		&\left|(\nabla_N \overline{L})\rho\right|=\left|g(\nabla_N \overline{L}, \nabla\rho)\right|=\left|g(g(\nabla_N \overline{L}, \overline{L})\overline{L}+g(\nabla_N \overline{L}, \overline{N})\overline{N}, \nabla\rho)\right|\leq C\|\nabla\rho\|\\
		&\left|g(\nabla_N\nabla\rho, \nabla_{\overline{L}} L)\right|=\left|g(\nabla_N\nabla\rho, g(\nabla_{\overline{L}} L, L)L)\right|\leq C\left|g(\nabla_N\nabla\rho, L)\right|\\
		&\left|g(\nabla_L\nabla\rho, \nabla_{\overline{N}}L)\right|=\left|g(\nabla_L\nabla\rho, g(\nabla_{\overline{N}}L, N)N+g(\nabla_{\overline{N}}L, L)L)\right|\leq C\left|\Hessian_\rho(L, N)\right|\\
		&\left|g(\nabla_{[N, L]}\nabla\rho, L)\right|=\left|g([N, L], L)g(\nabla_{L}\nabla\rho, L)+g([N, L], N)g(\nabla_{N}\nabla\rho, L)\right|\leq C\left|\Hessian_\rho(L, N)\right|.
		\end{split}\]
		
		Hence,
		\[\frac{1}{1-\eta}\leq2+\frac{3.5C\frac{\left|\Hessian_\rho(L, N)\right|}{\|\nabla\rho\|}-\frac{\Re Lg(\nabla_N\nabla\rho, L)}{2\|\nabla\rho\|}}{\left(\frac{|\Hessian_\rho(L, N)|}{\|\nabla\rho\|}\right)^2}.\]
		
		%
		%
		
		We look into \[\begin{split}
		-\frac{\frac{\Re Lg(\nabla_N\nabla\rho, L)}{\|\nabla\rho\|}}{\left(\frac{|\Hessian_\rho(L, N)|}{\|\nabla\rho\|}\right)^2}&=-\Re \frac{\|\nabla\rho\|Lg(\nabla_N\nabla\rho, L)}{\left|\Hessian_\rho(N, L)\right|^2}\\&\leq \left|\frac{\|\nabla\rho\|Lg(\nabla_N\nabla\rho, L)}{\left(\Hessian_\rho(N, L)\right)^2}\right|\\&\leq \left|\frac{\|\nabla\rho\|Lg(\nabla_N\nabla\rho, L)-(L\|\nabla\rho\|)g(\nabla_N\nabla\rho, L)}{\left(\Hessian_\rho(N, L)\right)^2}\right|+\left|\frac{L\|\nabla\rho\|}{\Hessian_\rho(N, L)}\right|\\&=\left|L\left(\frac{\|\nabla\rho\|}{\Hessian_\rho(N, L)}\right)\right|+\left|\frac{L\|\nabla\rho\|}{\Hessian_\rho(N, L)}\right|.
		\end{split}\]
		
		On $\partial\Omega$
		\[\begin{split}
		\frac{\overline{L}\|\nabla\rho\|}{\Hessian_\rho(N, L)}=\frac{(\overline{L}\psi)(N\delta)+\overline{L}(N\delta)}{(N\delta)(\overline{L}\psi)+\Hessian_\delta(N, L)}=&1+\frac{\overline{L}(N\delta)-\Hessian_\delta(N, L)}{(N\delta)(\overline{L}\psi)+\Hessian_\delta(N, L)}\\=&
		1+\frac{(\nabla_{\overline{L}}N)\delta}{(N\delta)(\overline{L}\psi)+\Hessian_\delta(N, L)}\\=&
		1+\frac{e^\psi(\nabla_{\overline{L}}N)\delta}{e^\psi(N\delta)(\overline{L}\psi)+e^\psi\Hessian_\delta(N, L)},
		\end{split}
		\]			
		which implies,
		\[\left|\frac{\overline{L}\|\nabla\rho\|}{\Hessian_\rho(N, L)}\right|\leq 1+\frac{C\|\nabla\rho\|}{2\left|\Hessian_\rho(N, L)\right|}\]		
		and 			
		\[\frac{1}{1-\eta}\leq 2.5+3.75C\frac{\|\nabla\rho\|}{\left|\Hessian_\rho(N, L)\right|}+0.5\left|L\left(\frac{\|\nabla\rho\|}{\Hessian_\rho(N, L)}\right)\right|.\]
		
		If the Diederich--Forn\ae ss index is $\eta_0$, the inequality above holds for any $\eta<\eta_0$. Hence we can find a sequence of smooth defining functions $\rho_j$ such that \[ 2.5+3.75C\frac{\|\nabla\rho_j\|}{\left|\Hessian_{\rho_j}(N, L)\right|}+0.5\left|L\left(\frac{\|\nabla\rho_j\|}{\Hessian_{\rho_j}(N, L)}\right)\right|\rightarrow \frac{1}{1-\eta_0},\] uniformly as $j\rightarrow\infty$.
		
		In case that the Diederich--Forn\ae ss index is 1, by the preceding inequality, we have that 
		\[\frac{\|\nabla\rho_j\|}{\left|\Hessian_{\rho_j}(N, L)\right|}+\left|L\left(\frac{\|\nabla\rho_j\|}{\Hessian_{\rho_j}(N, L)}\right)\right|\rightarrow\infty,\] uniformly as $j\rightarrow\infty$.
	\end{proof}
	
	We let $\rho=\delta e^\psi$. After a tedious computation we find that on $\partial\Omega$,

	\[\begin{split}
	&\frac{1}{2}L\left(\frac{\|\nabla\rho\|}{\Hessian_{\rho}(N, L)}\right)\\=&\frac{\left(LN\delta\right)\left((N\delta)(\overline{L}\psi)+g(\nabla_N\nabla\delta, L)\right)-\left(N\delta\right)\left((LN\delta)(\overline{L}\psi)+(N\delta)(L\overline{L}\psi)+L(g(\nabla_N\nabla\delta, L))\right)}{\left((N\delta)(\overline{L}\psi)+g(\nabla_N\nabla\delta, L)\right)^2}\\
	=&\frac{\left(LN\delta\right)g(\nabla_N\nabla\delta, L)-\left(N\delta\right)\left((N\delta)(L\overline{L}\psi)+L(g(\nabla_N\nabla\delta, L))\right)}{\left((N\delta)(\overline{L}\psi)+g(\nabla_N\nabla\delta, L)\right)^2}
	\end{split}\]
	
	Since $N\delta=\frac{1}{2}\|\nabla\delta\|=\frac{1}{2}$, $L(N\delta)=0$ and the equality above can be written as
	
	\[\begin{split}
	L\left(\frac{\|\nabla\rho\|}{\Hessian_{\rho}(N, L)}\right)&=\frac{-\left(\frac{1}{2}L\overline{L}\psi+L(g(\nabla_N\nabla\delta, L))\right)}{\left(\frac{1}{2}\overline{L}\psi+g(\nabla_N\nabla\delta, L)\right)^2}\\&=L\left(\frac{1}{\frac{1}{2}\overline{L}\psi+g(\nabla_N\nabla\delta, L)}\right)
	\end{split}
	\]
	
	We now compute 
	\[\frac{\|\nabla\rho\|}{\Hessian_{\rho}(N, L)}=\frac{1}{\frac{1}{2}\overline{L}\psi+g(\nabla_N\nabla\delta, L)}\]
	and 
	\[L\left(\frac{1}{\frac{1}{2}\overline{L}\psi+g(\nabla_N\nabla\delta, L)}\right)=-\frac{\frac{1}{2}L\overline{L}\psi+Lg(\nabla_N\nabla\delta, L)}{\left(\frac{1}{2}\overline{L}\psi+g(\nabla_N\nabla\delta, L)\right)^2}.\]
	
	Based on the simplification of the torsion $\frac{\|\nabla\rho\|}{\Hessian_{\rho}(N, L)}$, we derive the following corollary.
	
	\begin{corollary}
		Let $\Omega$ be a bounded pseudoconvex domain with smooth boundary in $\mathbb{C}^2$ and $\Sigma$ be the Levi-flat set of $\partial\Omega$. If the Diederich--Forn\ae ss index is $\eta_0$, then for any $\eta<\eta_0$, there exists a smooth functions $\psi$ such that \[ 2.5+\frac{3.75C}{\left|\frac{1}{2}\overline{L}\psi+g(\nabla_N\nabla\delta, L)\right|}+0.5\left|L\left(\frac{1}{\frac{1}{2}\overline{L}\psi+g(\nabla_N\nabla\delta, L)}\right)\right|\geq \frac{1}{1-\eta}.\]
		
		Particularly, if the Diederich--Forn\ae ss index is 1, then there exists a sequence of smooth functions $\psi_j$ so that \[\frac{1}{\frac{1}{2}\overline{L}\psi_j+g(\nabla_N\nabla\delta, L)}\] diverges to $\infty$ on $\Sigma$ in the holomorphic $C^1$ norm.
	\end{corollary}

	\section{The second necessary condition of Diederich--Forn\ae ss index}\label{2nd}

	\begin{proof}[Proof of Theorem \ref{second condition}]
		Starting from Theorem \ref{rough},
		\[\frac{1}{1-\eta}\leq\frac{-2\Re\left(\Hessian_\rho (L, N)\cdot\frac{\overline{L}\rho}{-\rho}\cdot N\rho\right)+\frac{\Hessian_\rho (L, L)}{-\rho}\left|N(\rho)\right|^2}{\left|\Hessian_\rho (L, N)\right|^2}.\]
		Recall that $N\rho=\overline{N}\rho$ implies $N\rho=\frac{1}{2}\|\nabla\rho\|$ and \[N(\overline{L}\rho)-(\nabla_N\overline{L})\rho=\Hessian_\rho(N, L).\] As a consequence,
		\[
		-2\Re\left(\Hessian_\rho(L, N)\cdot \lim\limits_{z\to\partial\Omega}\frac{\overline{L}\rho}{-\rho}\cdot N\rho\right)=2|\Hessian_\rho(L, N)|^2+2\Re\left(\Hessian_\rho(L, N)\cdot(\nabla_N \overline{L})\rho\right).
		\]
		Here, we use the computation \[-\lim\limits_{z\to\partial\Omega}\frac{\overline{L}\rho}{-\rho}=\lim\limits_{z\to\partial\Omega}\frac{\overline{L}\rho\vert_z-\overline{L}\rho\vert_{z_0}}{\rho\vert_z-\rho\vert_{z_0}}=\frac{(N+\overline{N})\overline{L}\rho}{(N+\overline{N})\rho}\bigg|_{z_0}=\frac{N(\overline{L}\rho)}{N\rho}\bigg|_{z_0},\] where $z_0$ is the boundary point such that $z_0-z$ is normal to $\partial\Omega$.
		Observe that on $\partial\Omega$
		\[(\nabla_N\overline{L})\rho=e^\psi(\nabla_N\overline{L})\delta=e^\psi(-\Hessian_\delta(N, L)+N(\overline{L}\delta))=-e^\psi\Hessian_\delta(N, L).\]
		This is because that $L\delta=\overline{L}\delta=0$ at all points in $\overline{\Omega}$, not only at the boundary. Moreover, \[\begin{split}
		\Hessian_\rho(L, N)=&g(\nabla_L\nabla\rho, N)=g(\nabla_L\nabla(\delta e^\psi), N)\\=&e^\psi L(\delta)\overline{N}\psi+\delta g(\nabla_L\nabla{e^\psi}, N)+L(e^\psi)\overline{N}\delta+e^\psi g(\nabla_L\nabla\delta,N)\\=&\delta g(\nabla_L\nabla{e^\psi}, N)+e^\psi L(\psi)\overline{N}\delta+e^\psi g(\nabla_L\nabla\delta,N).
		\end{split}\]
		and on $\partial\Omega$, we have that \[\Hessian_\rho(L, N)=e^\psi L(\psi)\overline{N}\delta+e^\psi g(\nabla_L\nabla\delta,N).\]
		Hence, 	
		\[\begin{split}
		2\Re\left(\Hessian_\rho(L, N)\cdot(\nabla_N \overline{L})\rho\right)=&-2\Re\left(e^{2\psi}\left((L\psi)(\overline{N}\delta)+\Hessian_\delta(L, N)\right)\left(\Hessian_\delta(N, L)\right)\right)\\=&-2e^{2\psi}\Re\left((L\psi)(\overline{N}\delta)\Hessian_\delta(N, L)\right)-2e^{2\psi}\left|\Hessian_\delta(L, N)\right|^2
		\end{split}\]
		
		We obtain,
		
		\[\begin{split}
		\frac{1}{1-\eta}-2\leq&\frac{2\Re\left(\Hessian_\rho(L, N)\cdot(\nabla_N \overline{L})\rho\right)+\frac{\Hessian_\rho (L, L)}{-\rho}\left|N(\rho)\right|^2}{\left|\Hessian_\rho (L, N)\right|^2}\\
		=&\frac{-2e^{2\psi}\Re\left((L\psi)(\overline{N}\delta)\Hessian_\delta(N, L)\right)-2e^{2\psi}\left|\Hessian_\delta(L, N)\right|^2}{\left|\Hessian_\rho (L, N)\right|^2}\\&-\frac{\|\nabla\rho\|\Re(Lg(\nabla_N\nabla\rho, L))}{2\left|\Hessian_\rho (L, N)\right|^2}\\&-\frac{\|\nabla\rho\|\Re\left(-g(\nabla_N\nabla\rho, \nabla_{\overline{L}} L)+g(\nabla_L\nabla\rho, \nabla_{\overline{N}}L)+g(\nabla_{[N, L]}\nabla\rho, L)\right)}{2\left|\Hessian_\rho (L, N)\right|^2}\\=&\frac{-2e^{2\psi}\Re\left((L\psi)(\overline{N}\delta)\Hessian_\delta(N, L)\right)-2e^{2\psi}\left|\Hessian_\delta(L, N)\right|^2}{\left|\Hessian_\rho (L, N)\right|^2}\\&-\frac{\|\nabla\rho\|\Re(Lg(\nabla_N\nabla\rho, L))}{2\left|\Hessian_\rho (L, N)\right|^2}\\&-\frac{\|\nabla\rho\|\Re\left(-g(\nabla_N\nabla\rho, \nabla_{\overline{L}} L)+g(\nabla_{[N, L]}\nabla\rho, L)\right)}{2\left|\Hessian_\rho (L, N)\right|^2}\\&-\frac{\|\nabla\rho\|\Re\left(g(\nabla_L\nabla\rho, \nabla_{\overline{N}}L)\right)}{2\left|\Hessian_\rho (L, N)\right|^2}
		\end{split}.\]
		
		We observe that \[-\Hessian_\delta(L, N)=(\nabla_{\overline{N}}L)\delta=g(\nabla_{\overline{N}}L, \sqrt{2}N)(\sqrt{2}N\delta)+g(\nabla_{\overline{N}}L, \sqrt{2}L)(\sqrt{2}L\delta)=g(\nabla_{\overline{N}}L, N),\]
		where $2N\delta=\|\nabla\delta\|=1$ because that $\delta$ is a distance function. Hence the last term is 
		\[\frac{\|\nabla\rho\|\Re\left(g(\nabla_L\nabla\rho, \sqrt{2}N)\sqrt{2}\Hessian_\delta(N, L)\right)}{2\left|\Hessian_\rho (N, L)\right|^2}=\frac{e^{2\psi}\left|\Hessian_\delta(L, N)\right|^2+e^{2\psi}\Re\left((L\psi)(\overline{N}\delta)\Hessian_\delta(N, L)\right)}{\left|\Hessian_\rho (N, L)\right|^2}.\]
		
		Thus, we have that,
		
		\[\begin{split}
		\frac{1}{1-\eta}-2\leq&\frac{-e^{2\psi}\Re\left((L\psi)(\overline{N}\delta)\Hessian_\delta(N, L)\right)-e^{2\psi}\left|\Hessian_\delta(L, N)\right|^2}{\left|\Hessian_\rho (L, N)\right|^2}\\&-\frac{e^\psi\Re(Lg(\nabla_N\nabla\rho, L))}{2\left|\Hessian_\rho (L, N)\right|^2}\\&-\frac{e^\psi\Re\left(-g(\nabla_N\nabla\rho, \nabla_{\overline{L}} L)+g(\nabla_{[N, L]}\nabla\rho, L)\right)}{2\left|\Hessian_\rho (L, N)\right|^2}\\=&\frac{-\Re\left(\frac{1}{2}(L\psi)\Hessian_\delta(N, L)\right)-\left|\Hessian_\delta(L, N)\right|^2}{\left|\frac{1}{2}\overline{L}\psi+\Hessian_\delta(N, L)\right|^2}\\&-\frac{\frac{1}{4}|L\psi|^2+\frac{1}{4}\Re (L\overline{L}\psi)+\frac{1}{2}\Re((L\psi)\Hessian_\delta(N,L)+\frac{1}{2}L\Hessian_\delta(N,L))}{\left|\frac{1}{2}\overline{L}\psi+\Hessian_\delta (L, N)\right|^2}\\&-\frac{e^\psi\Re\left(-g(\nabla_N\nabla\rho, \nabla_{\overline{L}} L)+g(\nabla_{[N, L]}\nabla\rho, L)\right)}{2\left|\Hessian_\rho (L, N)\right|^2}\\\leq&\frac{C_1}{\left|\frac{1}{2}\overline{L}\psi+\Hessian_\delta(N,L)\right|}+\frac{-\frac{1}{4}|L\psi|^2-\Re\left((L\psi)\Hessian_\delta(N, L)\right)-\left|\Hessian_\delta(L, N)\right|^2}{\left|\frac{1}{2}\overline{L}\psi+\Hessian_\delta(N, L)\right|^2}\\&+\frac{-\frac{1}{4}\Re (L\overline{L}\psi)-\frac{1}{2}\Re(L\Hessian_\delta(N,L))}{\left|\frac{1}{2}\overline{L}\psi+\Hessian_\delta (L, N)\right|^2}+\frac{\Re\left(\frac{1}{4}\nabla_{\overline{L}}L(\psi)+\frac{1}{2}\Hessian_\delta(N, \nabla_{\overline{L}}L)\right)}{\left|\frac{1}{2}\overline{L}\psi+\Hessian_\delta (L, N)\right|^2}\\=&\frac{C_1}{\left|\frac{1}{2}\overline{L}\psi+\Hessian_\delta(N,L)\right|}-1-\frac{\frac{1}{4}\Hessian_\psi(L, L)}{\left|\frac{1}{2}\overline{L}\psi+\Hessian_\delta (L, N)\right|^2}-\frac{\frac{1}{2}\Re(L\Hessian_\delta(N,L)-\Hessian_\delta(N, \nabla_{\overline{L}}L))}{\left|\frac{1}{2}\overline{L}\psi+\Hessian_\delta (L, N)\right|^2}.
		\end{split}\]
		
		This implies
		\[\frac{1}{1-\eta}-1\leq\frac{C_2-\frac{1}{4}\Hessian_\psi(L, L)}{|\frac{1}{2}\overline{L}\psi+\Hessian_\delta(N, L)|^2}+\frac{C_1}{|\frac{1}{2}\overline{L}\psi+\Hessian_\delta(N, L)|},\]
		where \[\label{C1}
		C_1:=2C+2\max_{\Sigma}|\Hessian_\delta(N, L)|\]
		and \[\label{C2}
		C_2:=\frac{1}{2}\max_{\Sigma}\left(\Re(-L\Hessian_\delta(N,L)+\Hessian_\delta(N, \nabla_{\overline{L}}L))\right).\]
	\end{proof}

	We can also refine the second necessary condition above in order to obtain the bounds of $|\frac{1}{2}\overline{L}\psi+\Hessian_\delta(N, L)|$. More precisely, we have the improved second necessary condition. We are ready to prove it.

	\begin{proof}[Proof of Theorem \ref{improv2nd}]
		We look at \[\frac{1}{1-\eta}-1\leq\frac{C_2-\frac{1}{4}\Hessian_\psi(L, L)}{|\frac{1}{2}\overline{L}\psi+\Hessian_\delta(N, L)|^2}+\frac{C_1}{|\frac{1}{2}\overline{L}\psi+\Hessian_\delta(N, L)|}\]
		
		If the Diederich--Forn\ae ss index is 1, the condition becomes \[\frac{C_2-\frac{1}{4}\Hessian_{\psi_n}(L, L)}{|\frac{1}{2}\overline{L}\psi_n+\Hessian_\delta(N, L)|^2}+\frac{C_1}{|\frac{1}{2}\overline{L}\psi_n+\Hessian_\delta(N, L)|}\rightarrow\infty\] as $n\rightarrow\infty$. In other words, we can find a subsequence $\psi_n$ so that \[\frac{C_2-\frac{1}{4}\Hessian_{\psi_n}(L, L)}{|\frac{1}{2}\overline{L}\psi_n+\Hessian_\delta(N, L)|^2}+\frac{C_1}{|\frac{1}{2}\overline{L}\psi_n+\Hessian_\delta(N, L)|}>n\] for all $n\in\mathbb{N}$. The preceding inequality is equivalent to 
		\[n\left|\frac{1}{2}\overline{L}\psi_n+g(\nabla_N\nabla\delta, L)\right|^2-C_1\left|\frac{1}{2}\overline{L}\psi_n+g(\nabla_N\nabla\delta, L)\right|+\frac{1}{4}\Hessian_{\psi_n}(L, L)-C_2<0\] for all $n\in\mathbb{N}$ which implies,
		\[C_1^2+n\left(4C_2-\Hessian_{\psi_n}(L, L)\right)\geq 0\] and
		\[0\leq \left|\frac{1}{2}\overline{L}\psi_n+g(\nabla_N\nabla\delta, L)\right|<\frac{C_1+\sqrt{C_1^2+n\left(4C_2-\Hessian_{\psi_n}(L, L)\right)}}{2n}\]for all $n\in\mathbb{N}$. This inequality also includes the case that \[\left|\frac{1}{2}\overline{L}\psi_n+g(\nabla_N\nabla\delta, L)\right|=0.\] 
		
		Hence, we obtain \begin{equation}\label{upper}\begin{split}
		\left|\frac{1}{2}\overline{L}\psi_n+g(\nabla_N\nabla\delta, L)\right|&<\frac{C_1+\sqrt{n}\sqrt{\frac{C_1^2}{n}+\left(4C_2-\Hessian_{\psi_n}(L, L)\right)}}{2n}\\&\leq\frac{C_1+\sqrt{n}(1+\frac{C_1^2}{n}+\left(4C_2-\Hessian_{\psi_n}(L, L)\right))}{2n}.
		\end{split}		
		\end{equation}
		For the last inequality, we used the Cauchy-Schwarz inequality.
	\end{proof}

	\section{An application of the second necessary condition}

	From the proof of the second necessary condition, the reader can see that $L$ can be any holomorphic vector field with a constant norm.

	\begin{proof}[Proof of Theorem \ref{app}]
		Firstly, since $L$ is a unit holomorphic tangent vector on $S$, we observe that $\Hessian_{\psi}(L, L)=\Delta\psi$ for any smooth function $\psi$ defined on $S$. Here $\Delta$ is defined with induced Euclidean metric from $\mathbb{C}^2$. 
		
		We assume the Diederich--Forn\ae ss index is 1. By (\ref{upper}), we integrate both sides over the closed Riemann surface $\Sigma$ and have that \[\int_{\Sigma}\left|\frac{1}{2}\overline{L}\psi_n+g(\nabla_N\nabla\delta, L)\right|\, dS\leq\int_{\Sigma}\frac{C_1+\sqrt{n}(1+\frac{C_1^2}{n}+\left(4C_2-\Delta\psi_n\right))}{2n}\,dS,\] where $\,dS$ is the surface element. Since $S$ has no boundary, by Stokes theorem, we have that 
		\[\int_{\Sigma}\Delta\psi_n\,dS=0.\] This gives that 
		\[\begin{split}
		\int_{\Sigma}\left|\frac{1}{2}\overline{L}\psi_n+g(\nabla_N\nabla\delta, L)\right|\, dS\leq&\int_{\Sigma}\frac{C_1+\sqrt{n}(1+\frac{C_1^2}{n}+4C_2)}{2n}\,dS-\int_{\Sigma}\frac{\Delta\psi_n}{2\sqrt{n}}\,dS\\=&\int_{\Sigma}\frac{C_1+\sqrt{n}(1+\frac{C_1^2}{n}+4C_2)}{2n}\,dS\\=&\text{Area}(\Sigma)\left(\frac{C_1+\sqrt{n}(1+\frac{C_1^2}{n}+4C_2)}{2n}\right).
		\end{split}\] 
		Let $n$ goes to $\infty$ we have that there exist a sequence $\psi_n$ so that 
		\[\int_{\Sigma}\left|\frac{1}{2}\overline{L}\psi_n+g(\nabla_N\nabla\delta, L)\right|\, dS\rightarrow 0.\]
		This completes the proof.
	\end{proof}

	\bigskip
	\bigskip
	\noindent {\bf Acknowledgments}. The author thanks to Dr. Anne-Katrin Gallagher for her patience and instruction on improving this paper. The author also thanks to Dr. Steven Krantz for his nice comments. The author also thank to Dr. Marco Peloso, Dr. Bun Wong, Dr. Yuan Yuan, Dr. Qi S. Zhang, Xinghong Pan, Dr. Lihan Wang and Dr. Meng Zhu for fruitful conversations. 
	
	\printbibliography
	
\end{document}